\documentclass[10pt,a4paper, leqno]{article}
\usepackage[utf8]{inputenc}
\usepackage{amsmath}
\usepackage{graphicx}%
\usepackage{amsfonts}
\usepackage{amssymb}
\usepackage{comment}
\usepackage{hyperref}
\usepackage{enumitem,pstricks,xy,pst-node}
\xyoption{all}
\usepackage{amsthm}
\usepackage{longtable}
\usepackage{caption,subcaption} 
\usepackage{booktabs} 
\usepackage[color]{changebar} 
\cbcolor{red} 
\usepackage{mathrsfs} 
\usepackage{verbatim} 
\usepackage[disable] 
{todonotes}
\usepackage[normalem]{ulem} 
\usepackage{caption} 
\usepackage{nicefrac} 
\usepackage{enumitem} 
\newtheorem{theorem}{Theorem}[section]

\newtheorem{corollary}[theorem]{Corollary}

\theoremstyle{definition}

\newtheorem{remark}[theorem]{Remark}

\theoremstyle{remark}

\DeclareMathOperator{\Hb}{H}
\DeclareMathOperator{\hsm}{h}
\DeclareMathOperator{\pgl}{PGL}

\newcommand{\codim}{\operatorname{codim}}

\newcommand{\rank}{\operatorname{rank}}

\newcommand{\ladi}{\begin{lastadd}}
\newcommand{\ladf}{\end{lastadd}}
\newcommand{\lrei}{\begin{lastrem}}
\newcommand{\lref}{\end{lastrem}}
\newenvironment{lastadd}
{\cbstart\color{red}}
{\todo{red to remove}\cbend}
\newenvironment{lastrem}
{\cbstart\color{yellow}}
{\cbend}

\author{Hanieh Keneshlou
\and Frank-Olaf Schreyer 
}
\title{The unirational components of the strata of genus $11$ curves with several pencils of degree $6$ in $\mathcal{M}_{11}$}
\begin{document}
\maketitle
\begin{abstract}
We show that the strata $ \mathcal{M}_{11,6}(k) \subset \mathcal{M}_{11} $ of $ 6-$gonal curves of genus $ 11 $, equipped with $k$ mutually independent and type I pencils of degree six, have a unirational irreducible component for $5\leq k\leq 9$. The unirational families arise from degree $ 9 $ plane curves with $ 4 $ ordinary triple and $ 5 $ ordinary double points that dominate an irreducible component of expected dimension. We will further show that the family of degree $ 8 $ plane curves with $ 10 $ ordinary double points covers an irreducible component of excess dimension in $ \mathcal{M}_{11,6}(10) $. 
\end{abstract}
\section*{Introduction}
Let $ C $ be a smooth irreducible $ d-$gonal curve of genus $ g $ defined over an algebraically closed field $ \mathbb{K} $. Recall that by definition of gonality, there exists a $ g^{1}_{d} $ but no $ g^{1}_{d-1} $ on $ C $. It is well-known that $ d\leq [\frac{g+3}{2}]$ with equality for general curves. In a series of papers (\cite{coppens4},\cite{coppens98},\cite{coppens3}, \cite{coppens1}, \cite{coppens2}) Coppens studied the number of pencils of degree $ d $ on $ C $, for various $ d $ and $ g $. For low gonalities up to $ d=5 $, the problem is intensively studied for almost all possible genera. For $ 6-$gonal curves, Coppens has settled the problem only for genera $ g\geq 15 $.\\
\indent
In this paper, we focus on $ 6-$gonal curves of genus $ g=11 $. The motivation for our choice of genus $ 11 $ was the question asked by Michael Kemeny, whether any smooth curve of genus $ 11 $ carrying at least six pencils $ g^{1}_{6}$'s, comes from degree $ 8 $ plane curves with $ 10 $ ordinary double points, where the pencils are cut out by the pencil of lines through each of the singular points. More precisely, there exists no smooth curve of genus $ 11 $ possessing exactly $ 6,7,8 $ or $ 9 $ pencils of degree six. We will show the answer to this question is negative.

Let $ \mathcal{M}_{11,6}(k)\subset \mathcal{M}_{11}$ be the stratum of smooth $ 6-$gonal curves of genus $ 11 $, equipped with exactly $ k $ mutually independent\footnote{Two pencils $ g_1,g_2 $ of degree $ d $ on a smooth curve $ C $ are called independent the corresponding map gives a birational model of $ C $ inside $ \mathbb{P}^{1}\times \mathbb{P}^{1} $} $ g^{1}_{6} $'s of type I\footnote{A base point free pencil $ g^{1}_{d} $ on a smooth  curve $C$ is called of type I if $ \dim \vert 2g^{1}_{d}\vert=2 $. Type I pencils are exactly those that we should count with multiplicity 1.}. We first investigate the possible number of $ g^{1}_{6} $'s on a $ 6-$gonal curve of genus $ 11 $, and therefore the possible values of $ k $ for which $ \mathcal{M}_{11,6}(k) $ is non-empty. In \cite{sometopic}, Schreyer gave a list of conjectural Betti tables for canonical curves of genus $ 11 $. Related to our question and interesting for us is the Betti table of the following form
\begin{center}
\begin{tabular}{|c c c c c c c c c c}
 \hline
  $1$ & $.$ & $ . $& $.$ & $.$ & $.$& $ . $ &$ . $&$ . $&$ . $\\ 
  $.$ & $36$ & $ 160 $&$315$ & $288$ & $5k$& $ . $& $ . $& $ . $& $ . $\\ 
  $.$ &$.$ & $.$ & $ . $& $5k$ & $288$ &$315$ &$ 160 $&$ 36 $&$ . $\\  
  $.$ &$.$ & $.$ & $ . $& $.$ & $.$ &$.$ &$ . $&$ . $&$ 1 $             
\end{tabular}
\label{bettitable} 
\end{center}
where $ k $ is expected to have the values $ k=1,2,\ldots,10,12,20$. Although, in view of Green's conjecture \cite{green}, it is not clear that for a smooth canonical curve of genus $ 11 $ with Betti table as above, the number $ k $ can always be interpreted as the multiple number of pencils of degree six existing on the curve. Nonetheless, for $ k=1,2,\ldots,10,12,20 $ we can provide families of curves, whose generic element carries exactly $ k $ mutually independent pencils of type I. The critical Betti number in this case is $ \beta_{5,6}=\beta_{4,6}=5k $ as expected. Therefore, in this range the stratum $ \mathcal{M}_{11,6}(k) $ is non-empty.\\

The first natural question is then to ask about the geometry of the stratum $ \mathcal{M}_{11,6}(k) $, in particular about its unirationality.\\

For $ k=1 $, the corresponding stratum is the famous Brill-Noether divisor $ \mathcal{M}_{11,6} $ of $ 6-$gonal curves \cite{harrismumford}, which is irreducible and furthermore known to be unirational \cite{GeissUnirationality}. The stratum $  \mathcal{M}_{11,6}(2) $ is irreducible \cite{tyomkin}, and unirational such that a general element of $ \mathcal{M}_{11,6}(2) $ can be obtained from a model of bidegree $ (6,6) $ in $ \mathbb{P}^{1}\times \mathbb{P}^{1}$ with $ \delta=14 $ ordinary double points. In \cite{keneshlou} it has been also shown that $ \mathcal{M}_{11,6}(3) $ has a unirational irreducible component of expected dimension. A general curves lying on this component can be constructed via liaison in two steps from a rational curve in multiprojective space $ \mathbb{P}^{1}\times \mathbb{P}^{1}\times \mathbb{P}^{1}$.\\

Here we construct rational families of curves with additional pencils from plane curves of suitable degrees with only ordinary multiple points, as singularities. As the first significant result (Theorem \ref{degree9}), we will prove that for $ 5\leq k\leq9 $, the stratum $ \mathcal{M}_{11,6}(k)$ has a unirational irreducible component of expected dimension. A general curve lying on this component arises from a degree $ 9 $ plane model with $ 4 $ ordinary triple and $ 5 $ ordinary double points which contains $ k-5 $ points among the ninth fixed point of the pencil of cubics passing through the $ 4 $ triple and $ 4 $ chosen double points .

The key technique of the proof is to study the space of first order equisingular deformations of plane curves with prescribed singularities, as well as that of the first order embedded deformations of their canonical model. In fact, denoting by $ M $ the $ 5k\times 5k $ submatrix in the deformed minimal resolution corresponding to the general first order deformation family of a canonical curve $ C $ with Betti table as above, we use the condition $ M=0 $ to determine the subspace of the deformations with extra syzygies of rank $ 5k $. It turns out that for $  5\leq k\leq 9 $, and respectively $ k $ linearly independent linear forms $ l_1,\ldots, l_k $ in the free deformation parameters corresponding to a basis of $ T_C\mathcal{M}_{11} $, we have $ \det M=l_1^{5}\cdot \ldots \cdot l_k^{5} $. This implies that $\mathcal{M}_{11,6}(k) $ has an irreducible component of exactly codimension $ k $ inside the moduli space $  \mathcal{M}_{11}  $. Furthermore, let $ \mathcal{K}_{11} $ to be the locus of the curves $ C\in \mathcal{M}_{11} $ with extra syzygies, that is $ \beta_{5,6}\neq 0 $. It is known by Hirschowitz and Ramanan \cite{hirsch} that $ \mathcal{K}_{11} $ is a divisor, called the Koszul divisor, such that $ \mathcal{K}_{11}=5 \mathcal{M}_{11,6}  $. Thus, $ \mathcal{M}_{11,6}  $ at the point $ C $ is locally analytically the union of $ k $ smooth transversal branches.\\
We will then compute the kernel of the Kodaira-Spencer map and from that the rank of the induced differential maps, in order to show that the rational families of plane curves dominate this component.

By following the similar approach, we obtain our second main result. We show that the family of degree $ 8 $ plane curves with $ 10 $ ordinary double points covers an irreducible component of excess dimension in $  \mathcal{M}_{11,6}(10) $ (Theorem \ref{degree10}).

This paper is structured as follow. In section \ref{sec:familiesofcurves} we recall some basics of deformation theory for smooth and singular plane curves. In section \ref{sec:tangentspacecomputation} we deal with the computation of the tangent spaces to our parameter spaces and we continue by proving the main theorems on unirationality in section \ref{sec:injectivity}. In the last section \ref{sec:Further components}, using the syzygy schemes of the curves, we study the irreducibility of these strata.

Our results and conjectures rely on the computations and experiments, performed by the computer algebra system \textit{Macaulay2} \cite{m2} and uses the supporting functions in the packages \cite{kenesh} and \cite{kensch}.
\section*{\small{\textbf{Acknowledgement}}}
We would like to thank Michael Kemeny for discussing this question with us which was the motivation point of this work.
This work is a contribution to Project I.6 within the SFB-TRR 195  “Symbolic Tools in Mathematics and their Application” of the German Research Foundation (DFG).

\section{Planar model description}
\label{sec:Planar model description}
In this section, we describe families of plane curves of genus $11 $ carrying $ k=4,\ldots,10,12,20 $ pencils. In particualr, we give a model of genus 11 curve with infinitely many pencils, arised as the triple cover of an elliptic curve. Throughout this paper, to avoid iteration, a pencil is always of the degree six, unless otherwise mentioned, and several pencils on a curve are supposed to be mutually independent of type I.\\ 

We first deal with the construction of plane model for smooth curves of genus $ 11 $ with $ k=5,\ldots,9 $ pencils. Clearly, smooth curves of genus $ 11 $ with ten pencils can be constructed from a plane model of degree $ 8 $ with $ 10 $ ordinary double points in general position. The code provided by the function \texttt{random6gonalGenus11Curve10pencil} in \cite{kenesh}, uses this plane model to produce a random canonical curve of genus $ 11 $ with exactly $ 10g^{1}_{6}$'s. We remark that, although we further provide a method to produce curves with $ k=4,12 $ pencils, by dimension reasons the rational family obtained from these models may not cover any component of the corresponding stratum.\\

\noindent
\textsc{model of curves with $ 5\leq k\leq 9 $ pencils} \\

Let $ P_1,\ldots,P_4, Q_1,\ldots,Q_5 $ be general points in the projective plane $ \mathbb{P}^{2} $ and let $ \Gamma \subset \mathbb{P}^{2} $ be a plane curve of degree $ 9 $ with $ 4 $ ordinary triple points $ P_1,\ldots,P_4 $, and $ 5 $ ordinary double points $ Q_1,\ldots,Q_5 $. We note that, since an ordinary triple (resp. double) point in general position imposes six (resp. three) linear conditions, such a plane curve with these singular points exists as 
\[ \binom{9+2}{2}-6\cdot 4-3\cdot 5>0. \]
Blowing up these singular points
$$ \sigma: \widetilde{\mathbb{P}}^{2}=\mathbb{P}^{2}(\ldots, P_i,\ldots,Q_j,\ldots)\longrightarrow  \mathbb{P}^{2},$$
let $ C\subset \widetilde{\mathbb{P}}^{2} $ be the strict transformation of $ \Gamma $ on the blown up surface of $  \mathbb{P}^{2} $. Hence,
$$ C \sim 9H-\sum_{i=1}^{4}3E_{P_i}-\sum_{j=1}^{5}2E_{Q_j},$$
where $ H $ is the pullback of the class of a line in $ \mathbb{P}^{2} $, and $ E_{P_i} $ and $ E_{Q_j} $ denote the exceptional divisors of the blow up at the points $ P_i $ and $ Q_j $, respectively. 
By the genus-degree formula, $ C $ is a smooth curve of genus $ 11=\binom{9-1}{2}-4.3-5 $. Moreover, $ C $ admits five mutually independent pencils of type I. Indeed, for $ i=1, \ldots,4 $ the linear series $ \vert H-E_{P_i}\vert $, identified with the pencil of lines through the triple point $ P_i $ induces a base point free pencil $G_{i} $ on $ C $. As by adjunction, the canonical system $ \vert K_C\vert $ is cut out by the complete linear series
$$ \vert C+K_{\widetilde{\mathbb{P}}^{2}}\vert=\vert 6H-\sum_{i=1}^{4}2E_{P_i}-\sum_{j=1}^{5}E_{Q_j}\vert,$$ 
the linear series $ \vert K_{C}-2G_{i}\vert $ is cut out by
$$ \vert 4H-\sum_{i=1}^{4}2E_{P_i}-\sum_{j=1}^{5}E_{Q_j}+2E_{P_i}\vert.$$ 
Therefore, we have $\dim \vert K_{C}-2G_{i}\vert=0  $ and by Riemann--Roch $ \dim \vert 2G_{i}\vert=2 $.  Thus, the induced pencils from linear system of lines through each of the triple points are of type I. Furthermore, the linear series $ \vert 2H-\sum_{i=1}^{4}E_{P_i}\vert $ identified with the the pencil of conics through the four triple points induces an extra pencil $ G_5 $ on $C$. Similarly by adjunction, the corresponding linear system $ \vert K_{C}-2G_5\vert $ can be identified with the linear system of quadrics containing the double points.  We obtain $\dim \vert K_{C}-2G_5\vert=0  $, which then Riemann--Roch implies that $ \dim \vert 2G_5\vert=2 $. Hence, this gives another pencil of type I. In this way we obtain smooth curves of genus $ 11 $ having five pencils. \\

In order to get the model of curves with further pencils, we impose certain one dimensional conditions on the plane curve of degree $ 9 $ such that each condition gives exactly one extra $ g^{1}_{6} $.\\
 
For $ j=1,\ldots,5 $, let $ R_j $ be the ninth fix point of the pencil of cubics through the eight residual singular points by omitting $ Q_j $. The condition that $ R_j$ lies on the plane curves
imposes exactly one condition on linear series of degree $ 9 $ plane curves with $ 4 $ ordinary triple points at $ P_i $'s and $ 5 $ ordinary double points at $ Q_j $'s. On the other hand, the linear series $$ \vert 3H-\sum_{i=1}^{4}E_{P_i}-\sum_{j=1}^{5}E_{Q_j}+E_{Q_j}\vert $$
induces a pencil $ G^{\prime}_j $ of degree $ 7 $ with a fix point at $ R_j $. Therefore, by forcing the degree $ 9 $ plane curves to pass additionally through each $ R_j $, we obtain one further pencil of type I, given by $G^{\prime}_{j}-R_j$. This way, by choosing $0\leq m\leq 4$ points among $ R_1, \ldots,R_5 $, we get families of smooth curves of genus $ 11 $ possessing up to nine pencils. The function \texttt{random6gonalGenus11Curvekpencil} in \cite{kenesh} is an implementation of the above construction which produces a random canonical curve of genus $ 11 $ possessing $ 5\leq k\leq 9 $ pencils.
 \begin{figure}[!hbt]
    \centering
    \begin{subfigure}{0.498\linewidth} 
        \centering
        \includegraphics[width=\textwidth]{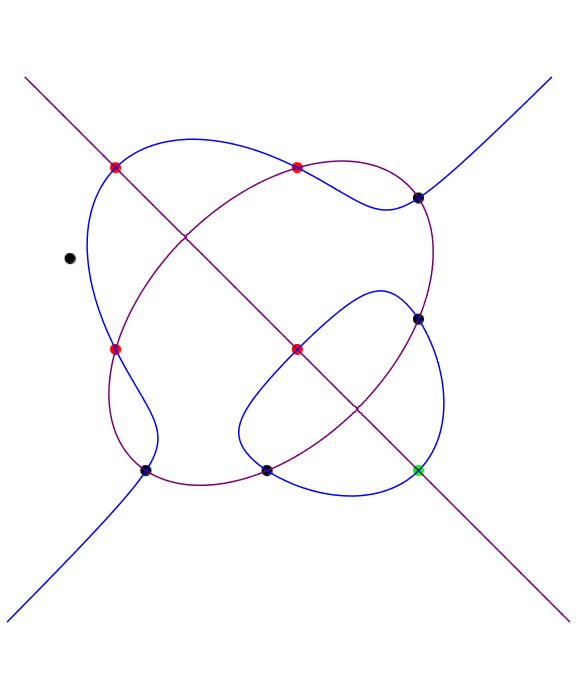}
        \caption*{Figure 4.1: Two cubics through 8 points by omitting one of the double points}
        \label{fig:1}
    \end{subfigure}
    \begin{subfigure}{0.5\linewidth}
        \centering
        \includegraphics[width=\textwidth]{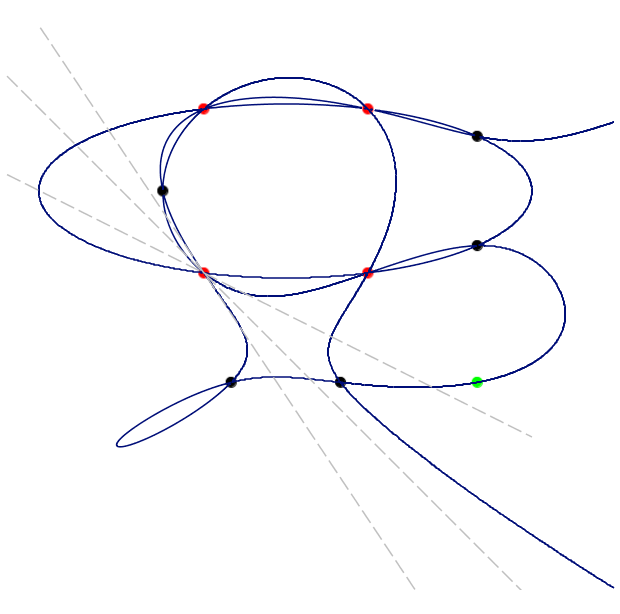}
        \caption*{Figure 4.2:  Plane curve of degree $ 9 $ with six pencils passing through the ninth fixed point}
        \label{fig:2}
    \end{subfigure}
\end{figure}
 \begin{remark}
\label{failur}
Although we expect that plane curves of degree $ 9 $ with singular points as above, passing through all the five fixed points $ R_1,\ldots,R_5 $, lead to the model of curves of genus $ 11 $ with ten pencils, our experimental computations show that such a curve is in general reducible. It is a union of a sextic and the unique cubic through the five double points and $ R_1,\ldots,R_5 $, which has further singular points than expected. Thus, our pattern fails to cover the case $ k=10 $.
\end{remark}
Our families of plane curves depend on expected number of parameters as desired. In fact, let 
\[\mathcal{V}_{9}^{4,5,m}:=\lbrace (\Gamma;P_1,\ldots,P_4, Q_1,\ldots, Q_5) \rbrace \subset \mathbb{P}^{N}\times (\mathbb{P}^{2})^{9}\]
denote the variety, where $ N=\binom{9+2}{2}-1 $ and $  \Gamma\subset \mathbb{P}^{2} $ is a plane curve of degree $ 9 $ with prescribed singular points passing through $0\leq m\leq 4$ points among $ R_1,\ldots,R_5 $ as above. As an ordinary triple (resp. double) point in general position imposes six (resp. three) linear conditions, we expect naively that each irreducible component of $ \mathcal{V}_{9}^{4,5,m}$  has dimension
\[\frac{9(9+3)}{2}+2\cdot 9-3\cdot 5-6\cdot 4-m=33-m.\]
Identifying the plane curves under automorphisms of $ \mathbb{P}^{2} $ reduces this dimension by $ 8=\dim \pgl(2) $. From Brill-Noether theory this fits to the expected dimension of the stratum $ \mathcal{M}_{11,6}(k)\subset \mathcal{M}_{11} $, of curves possessing $ k=m+5 $ pencils. In fact, denoting by $ \rho $ the Brill-Noether number, we have 
\[ \dim \mathcal{M}_{11,6}(k) \geq 3\cdot 11-3+(m+5)\rho(11,6,1)=25-m. \]

\noindent
\textsc{models of curves with $ k=4 $ pencils.} \\

Let $ P_1,P_2,P_3, Q_1,\ldots,Q_7$ be general points in the projective plane and  $ R $ be the ninth fix point of a pencil of cubics through eight points, obtained by omitting two of $ Q_i$'s. Then, the normalization of a general degree $ 9 $ plane curve with ordinary triple points at $ P_1,P_2,P_3 $ and ordinary double points at $ Q_1,\ldots,Q_7,R$ is a smooth curve of genus $ 11 $ that carries exactly $ k=4 $ pencils. In fact, the three pencils are induced from the pencil of lines through each of the triple points and the pencil of cubics through the eight points gives the extra $ g^{1}_{6} $. In \cite{kenesh}, this construction is implemented in the function \texttt{random6gonalGenus11Curve4pencil}.
\begin{remark}
\label{k4}
The number of parameters for the choice of ten points in the plane as above plus the dimension of the linear system of plane curves of degree $ 9 $ with ordinary triple points at $ P_1,P_2,P_3 $ and ordinary double points at $ Q_1,\ldots,Q_7,R$ amounts to $ 32 $ parameters. Therefore, modulo the isomorphisms of the projective plane, we obtain a family of smooth curves of genus $ 11 $ with exactly $ k=4 $ pencils and smaller dimension than $ 26 $, which is the expected dimension of $ \mathcal{M}_{11,6}(4) $. Thus, the rational family of curves obtained from this model cannot cover any component of $ \mathcal{M}_{11,6}(4) $.
\end{remark}
\noindent
\textsc{models of curves with $ k=12 $ pencils} \\

Let $ P_1,\ldots,P_{10} $ be general points in the projective plane and $  V_1 \subset \vert L\vert=\vert 4H-\sum_{i=1}^{10}E_{P_i}\vert $ be a pencil in the linear system of quartics passing through these points. Let $ q_1,\ldots,q_6 $ be the further fixed points of this pencil. Then, normalization of a degree $ 8 $ plane curve $ \Gamma $ with $ 10 $ ordinary double points $ P_1,\ldots,P_{10} $ and passing through $ q_1,\ldots,q_6 $, carries exactly twelve pencils. One has $10$ pencils cut out by the lines through each of the double points and also the $g^1_6$ cut out by $V_1$. Moreover, considering $ Q_1,\ldots,Q_6 $ to be the six moving points of a divisor in $ V_1 $, our experiments show that $ Q_1,\ldots,Q_6 $ are the extra fixed points of an another pencil $ V_2 \subset \vert L\vert $. Namely, there is a two dimentional vector space of quartics passing through $P_1,\ldots,P_{10}, Q_1,\ldots,Q_6$ cutting out the twelfth $g^1_6$. More precisely, let $\mathbb{P}^2\dashrightarrow \mathbb{P}^4$ be the rational map associated to $|L|$. The image of $\Gamma$ under this map is a curve $C$ of degree $12$, which is cut out by a
unique rank $4$ quadric hypersuface $Q$ on the determinantal image surface of $\mathbb{P}^2$. As the divisors of the linear series $|L|$ are cut out by the linear system of hyperplanes on $C$, and the six fixed points impose exactly three linearly independent
conditions on this linear series, they span a projective plane $\mathbb{P}^2\subset Q$ and they do
not lie on a conic. As $Q$ is isomorphic to the cone over $\mathbb{P}^1\times \mathbb{P}^1$, the projections
to each projective line naturally give two extra pencils.
In \cite{kenesh}, the function \texttt{random6gonalGenus11Curve12pencil} uses this method to produce a random canonical curve of genus $ 11 $ carrying exactly twelve pencils.\\

\noindent
\textsc{models of curves with $ k=20 $ pencils} \\

Let $ C $ be a smooth curve of genus $ 11 $ with a linear system $ g^3_{10} $. The space model of $ C $  has exactly twenty 4-secant lines which cut out the twenty pencils. A plane curve of degree $ 9 $ with $ 5 $ ordinary triple and $ 2 $ ordinary double points provides a model of such curves. Using this pattern, in \cite{kenesh}, the function \texttt{random6gonalGenus11Curv20pencil} gives model of genus $ 11 $ curves with $ 20g^1_6$'s. \\

\noindent
\textsc{models of curves with infinitely many pencils} \\

Let $ E\subset \mathbb{P}^{2} $ be a smooth plane cubic, and consider $ X_1:=E\times \mathbb{P}^{1}\subset \mathbb{P}^{2}\times \mathbb{P}^{1}$ as a hypersurface of bidegree $ (3,0) $ containing two random lines $ L_1, L_2 $ and four points $ P_1,\ldots,P_4 $. Choosing a random hypersurface $ X_2 $ of bidegree $ (3,3) $ with double points at $ P_i's $ and containing the two lines, we obtain the complete intersection $  X_1\cap X_2=C\cup L_1\cup L_2$, where $ C $ is the triple cover of the elliptic curve $ E $ of bi-degree $ (9,7) $ in $ \mathbb{P}^{2}\times \mathbb{P}^{1} $. Naturally, $ C $ admits infinitely many pencils which are cut out by the pencil of lines through random points of $ E $. In \cite{kenesh}, this algorithm is implemented in the function \texttt{random6gonalGenus11CurveInfinitepencil} and produces model of $ C $ of $ \deg(C)=16 $ in $ \mathbb{P}^{5} $. Considering the space of hyperplanes through three general points of $ C $, we obtain a $ g^2_{13} $. Using this linear series one can compute the plane model and from that the canonical model of $ C $ which leads into the Betti number $ \beta_{5,6}=25 $. With the same approach, and starting from three lines and the choice of two points, we obtain a genus $ 11 $ triple cover of an elliptic curve of bi-degree $ (9,6) $ whose canonical model has the Betti number $ \beta_{5,6}=30 $.

\section{Families of curves and their deformation}\label{sec:familiesofcurves}
To study the local geometry of parameter spaces introduced in the previous section, and also the strata of the smooth curves with several pencils, we study the space of the first order deformation of curves. This leads to the computation of the tangent space at the corresponding points in the moduli. We recall some basics on deformation theory for smooth and singular plane curves which can be found in the standard textbook \cite{Sernesi}.\\

Let $C\subset \mathbb{P}^{n} $ be a smooth curve and $ \mathcal{N}_{C/\mathbb{P}^{n}}=\mathcal{H}om_{\mathcal{O}_C}(\mathcal{I}/\mathcal{I}^{2},\mathcal{O}_C) $ denote the normal bundle  of $C$ in $ \mathbb{P}^{n} $. The space of global sections $ \Hb^{0}(C,\mathcal{N}_{C/\mathbb{P}^{n}}) $ parametrizes the set of first order embedded deformations of $ C$ in $ \mathbb{P}^{n}$. This is precisely the tangent space to the Hilbert scheme $ \mathcal{H}_{C/\mathbb{P}^{n}} $ of $ C $ inside $  \mathbb{P}^{n} $ (see \cite{Sernesi}, Theorem 3.2.12).

An important refinement of the embedded deformation of a smooth curve is consideration of flat families of curves inside a projective space having prescribed singularities, that is of families whose members have the same type of singularities in some specified sense. This leads to the notion of equisingularity.\\

Let $ \Gamma\subset \mathbb{P}^{2} $ be a singular plane curve. There exists an exact sequence of coherent sheaves on $ \Gamma $,
\[ 0\longrightarrow T_{\Gamma}\longrightarrow T_{\mathbb{P}^{2}}\vert_{\Gamma}\longrightarrow \mathcal{N}_{\Gamma/\mathbb{P}^{2}}\longrightarrow T^{1}_{\Gamma}\longrightarrow 0, \]
where the two middle sheaves are locally free, whereas the first one is not (see \cite{Sernesi}, Proposition 1.1.9). The sheaf $ T^{1}_{\Gamma} $ is the so-called cotangent sheaf, supported on the singular locus of $ \Gamma $. The \textit{equisingular normal sheaf} of $ \Gamma $ in $ \mathbb{P}^{2}$ is defined to be
\[  \mathcal{N}^{\prime}:=\ker[\mathcal{N}_{\Gamma/\mathbb{P}^{2}}\longrightarrow T^{1}_{\Gamma}],  \]
which describes deformations preserving the singularities of $ \Gamma $. In fact, the vector space $ \Hb^{0}(\Gamma,\mathcal{N}^{\prime}_{\Gamma/\mathbb{P}^{2}}) $ parameterizes the locally trivial first order deformations of $ \Gamma $ in $ \mathbb{P}^{2} $ having the prescribed singularities as $ \Gamma $ (See \cite{Sernesi},  Section 4.7.1). 
In particular, the equisingular normal bundle fits into the short exact sequence
\begin{equation}\label{short}
 0\longrightarrow \mathcal{O}_{\mathbb{P}^{2}}\longrightarrow \mathcal{I}(d)\longrightarrow \mathcal{N}^{\prime}_{\Gamma/\mathbb{P}^{2}}\longrightarrow 0,  
\end{equation}
where $ \mathcal{I} $ is the ideal sheaf locally generated by the partial derivatives of a local equation of $ \Gamma$, and the first injective map is defined by multiplication by an equation of $ \Gamma $ (See \cite{Sernesi}, page 55).
\section{The tangent space computation}\label{sec:tangentspacecomputation}
In this section, we compute the tangent space to the parameter space $ \mathcal{V}_{9}^{4,5,m} $ as well as that to the strata $ \mathcal{M}_{11,6}(k)\subset \mathcal{M}_{11} $. We further prove the existence of a component with expected dimension on both spaces.
\vskip 0.02cm
\begin{theorem}
\label{tangent}
For $ m=0,\ldots,4 $, the parameter space $\mathcal{V}_{9}^{4,5,m} $ has an irreducible component of expected dimension.
\end{theorem}
\begin{proof}
Let $ (\Gamma;P_1,\ldots,P_4, Q_1,\ldots, Q_5) \in \mathcal{V}_{9}^{4,5,m} $ be a point corresponding to a plane curve $ \Gamma:(f=0)\subset \mathbb{P}^{2}$ with prescribed singular points and passing through $ R_1,\ldots,R_m$. Assume $ x,y,z $ are the coordinates of the projective plane. Considering $\Gamma$ as a point in the parameter space $ \mathbb{P}^{{{9+2}\choose{2}}-1} $ of degree $ 9 $ plane curves, without loss of generality we can assume it lies in the affine chart, which does not contain the point $ (1:0:0) $. Moreover, to simplify our notations, we can assume all the distinguished points of $ \Gamma $ are in the open affine subset of $ \mathbb{P}^{2} $ defined by $ z=1 $. Thus, $ \Gamma $ is locally defined by $f=\sum_{u,v}a_{uv}x^{u}y^{v}$ such that $ a_{9,0}=1$, $ (x_i,y_i)$ for  $ \ 1\leq i\leq 9 $ are the affine coordinates of the singular points and $ (x'_l,y'_l) $ is the affine coordinate of $ R_l $. Therefore, in a neighbourhood of $ \Gamma $, the space $\mathcal{V}_{9}^{4,5,m} $ is the set of pairs $ (\bar{h}; S_1,\ldots,S_9) $ with $\bar{h}=\sum_{u,v}b_{uv}x^{u}y^{v}$, $ b_{9,0}=1 $ and $ S_i=(X_i,Y_i) $ for $ 1\leq i\leq 9 $, satisfying the following equations:
$$ R_{i,s,t}(\ldots,b_{uv},\ldots, X_j,Y_j,\ldots):=\frac{\partial \bar{h}}{\partial^{t}x\partial^{^{s-t}}y}(X_i,Y_i)=0,  $$
for $ 1\leq i\leq 4,\ s=0,1,2,\ t\in\lbrace 0,\ldots,s\rbrace  $,
$$ R'_{i,s,t}(\ldots,b_{uv},\ldots, X_j,Y_j,\ldots):=\frac{\partial \bar{h}}{\partial^{t}x\partial^{^{s-t}}y}(X_i,Y_i)=0,  $$
for $ \ 5\leq i\leq 9,\ s=0,1,\ t\in\lbrace 0,\ldots,s\rbrace  $ and 
$$ F_l:=(\sum_{u,v}b_{uv}x^{u}y^{v})(X'_l,Y'_l)=0, \ \ \forall\ 1\leq l\leq m,$$
where $ (X'_l,Y'_l) $ are the coordinates of $ m $ points among the fixed points. Then, the tangent space at $ \Gamma $ is the set of points $ (\bar{g}; T_1,\ldots,T_9) $ with $\bar{g}=\sum_{u,v}c_{uv}x^{u}y^{v}$, $ c_{9,0}=1, c_{uv}=a_{uv}+b_{uv} $ for $ u\neq 9 $ and $ T_i=(x_i+X_i,y_i+Y_i)$ for $ 1\leq i\leq 9 $, satisfying the following equations with indeterminate in $ \ldots,b_{uv},\ldots,X_j,Y_j,\ldots $: 
\begin{align*}
&\sum_{\substack{u,v\geq 0\\ u+v\leq 9\\ u\neq 9}}b_{uv}\frac{\partial R_{i,s,t}}{\partial b_{uv}}(\ldots,a_{uv},\ldots,x_i,y_i,\ldots)\\
& + \sum _{\alpha=0} ^{9}[X_\alpha\frac{\partial R_{i,s,t}}{\partial X_{\alpha}}(\ldots,a_{uv},\ldots,x_i,y_i,\ldots)+Y_{\alpha}\frac{\partial R_{i,s,t}}{\partial Y_{\alpha}}(\ldots,a_{uv},\ldots,x_i,y_i,\ldots)]=0
\end{align*}
for all $ 1\leq i\leq 4,\ s=0,1,2,\ t\in\lbrace 0,\ldots,s\rbrace  $, the same relation with $ R'_{i,s,t} $, for all $ \ 5\leq i\leq 9,\ s=0,1,\ t\in\lbrace 0,\ldots,s\rbrace  $ and 
$$ \sum _{\substack{u,v\geq 0\\ u+v\leq 9\\ u\neq 9}}b_{uv}\frac{\partial \bar{h}}{\partial b_{uv}}(x'_l,y'_l)=0, \ \ \ \forall\ 1\leq l\leq m.  $$  
In \cite{kenesh}, the code provided by the implemented function \texttt{verifyAssertion(1)} uses this method to compute the tangent space as the space of solutions to the above equations. Our computation of an explicit example for a randomly chosen point on $ \mathcal{V}_{9}^{4,5,m} $ shows that this space is of dimension $ 33-m$. Therefore, the irreducible component of $ \mathcal{V}_{9}^{4,5,m} $ containing that point is of expected dimension. 
\end{proof}
\begin{remark}
\label{cohomologygroup}
Let $ (\Gamma;P_1,\ldots,P_4, Q_1,\ldots, Q_5) \in \mathcal{V}_{9}^{4,5,0}$ be a point and let $ \Delta $ denote the singular locus of the corresponding plane curve with prescribed number of double and triple points. Via the first projection map 
$$p_1:\mathcal{V}_{9}^{4,5,0}\longrightarrow \mathcal{V}^{4,5}_{9}\subset \mathbb{P}^{N},$$
the variety $\mathcal{V}_{9}^{4,5,0} $ maps one-to-one to the Severi variety $ \mathcal{V}^{4,5}_{9} $, parametrizing the degree $ 9 $ plane curves with $ 4 $ ordinary triple points and $ 5 $ ordinary double points. This way, we can naturally denote $ \mathcal{V}_{9}^{4,5,0} $ by $ \mathcal{V}_{9}^{4,5} $ and identify the tangent space to $\mathcal{V}_{9}^{4,5,0}$ at $ \Gamma $ with the space of the first order deformation of $ \Gamma\in \mathcal{V}^{4,5}_{9}$. Thus, from the short exact sequence $\ref{short}$ we obtain
$$ T_{\Gamma}\mathcal{V}_{9}^{4,5}\cong \Hb^{0}(\mathbb{P}^{2},\mathcal{I}_{\Delta}(9))/\langle f\rangle, $$
where $ \langle f\rangle $ is the one-dimensional vector space generated by the defining equation of $ \Gamma $. Moreover, for $ m>0 $ the computed tangent space to $ \mathcal{V}_{9}^{4,5,m} $ at a random point as in Theorem \ref{tangent}, can be regarded as a subspace of such a vector space.
\end{remark}

\indent
Now we turn to the computation of the tangent space to the strata $ \mathcal{M}_{11,6}(k)$.\\

\noindent
Let $ C\subset\mathbb{P}^{10} $ be a smooth canonical curve with extra syzygies of rank $5k$ and
\begin{small}
\[
0\longleftarrow  S/I_{C} \longleftarrow S\xleftarrow{f} S(-2)^{36}\xleftarrow{\varphi_{1}} S(-3)^{160}\xleftarrow{\varphi_{2}}S(-4)^{315}
\xleftarrow{\varphi_{3}} \begin{matrix} 
S(-5)^{288}  \\ 
\oplus \\ 
 S(-6)^{5k} 
\end{matrix}
\xleftarrow{\varphi_{4}} \begin{matrix} 
S(-6)^{5k}\\
\oplus\\
S(-7)^{288} \end{matrix} 
\]
\end{small}
be the part of a minimal free resolution of $ C $, where $ S=\mathbb{K}[x_0,\ldots,x_{10}] $ is the coordinate ring of $ \mathbb{P}^{10}$, and $ f=(f_1,\ldots,f_{36}) $ is the minimal set of generators of the ideal $ I_{C} \subset S$. Consider the pullback to $C$ of the Euler sequence 
\begin{equation}
\label{euler}
 0\longrightarrow \mathcal{O}_{C}\longrightarrow \mathcal{O}_{C}(1)^{\oplus g}\longrightarrow T_{\mathbb{P}^{10}}\vert_{C}\longrightarrow 0. 
\end{equation}
From the long exact sequence of cohomologies, the dual vector space $ \Hb^{1}(C,T_{\mathbb{P}^{10}}\vert_{C})^{\vee} $ can be identified with the kernel of the Petri map
$$ \mu_{0}:\Hb^{0}(C,L)\otimes \Hb^{0}(C,\omega_{C}\otimes L^{-1})\longrightarrow \Hb^{0}(C,\omega_{C})$$
where $ L=\mathcal{O}_{C}(1)$. Therefore, we get
$ \Hb^{1}(C,T_{\mathbb{P}^{10}}\vert_{C})=0$
and from that, the induced long exact sequence of the normal exact sequence
$$ 0\longrightarrow T_{C}\longrightarrow T_{\mathbb{P}^{10}}\vert_{C}\longrightarrow \mathcal{N}_{C/\mathbb{P}^{10}}\longrightarrow 0,$$
reduces to the following short exact sequence
\begin{equation}
\label{kodariraspencer}
 0 \longrightarrow \Hb^{0}(C,T_{\mathbb{P}^{10}}\vert_{C})\longrightarrow \Hb^{0}(C,\mathcal{N}_{C/\mathbb{P}^{10}})\xrightarrow{\kappa} \Hb^{1}(C,T_{C})\longrightarrow 0,
\end{equation} 
where $ \kappa $ is the so-called Kodaira-Spencer map. More precisely, here we realize $ \Hb^{1}(C,T_{C}) $ as the tangent space to the moduli space $ \mathcal{M}_{11} $ at the point corresponding to $ C $, and $ \kappa $ as the induced map between the tangent spaces from the natural map $ \mathcal{H}_{C/\mathbb{P}^{10}}\longrightarrow \mathcal{M}_{11} $. We observe that by Serre duality 
  $$\Hb^{1}(C,T_{C}) \cong \Hb^{0}(C,\omega_{C}^{\otimes 2})^{\vee}. $$
Since we assume that the curve is canonically embedded, the sheaf $ \omega_{C}^{\otimes 2} $ is just the twisted sheaf $ \mathcal{O}_C(2) $. Hence, the cohomology
group above will be given by the quotient $ S_2/(I_C)_2 $ and thus $  \hsm^{1}(C,T_{C})=30$. As $ I_C $ is minimally generated by $ 36 $ generators, we can identify a basis of $ \Hb^{1}(C,T_{C}) $ with columns of a matrix of size $ 36\times 30 $ with entries in $ S_2/(I_C)_2 $, introducing $ 30 $ free deformation parameters  $ b_0,\ldots, b_{29} $. Let $ \bar{f}=f+f^{(1)} $ be the general first order family perturbing $ f $ defined by the general element of $ \Hb^{1}(C,T_{C}) $ and let
$$ \bar{S}\xleftarrow{\bar{f}} \bar{S}(-2)^{36} $$
be the corresponding morphism, where 
$$ \bar{S}=\mathbb{K}[b_0,\ldots,b_{29}]/(b_0,\ldots,b_{29})^{2} \otimes_{\mathbb{K}} S.$$ 
To find a lift $ \bar{\varphi_{1}}=\varphi_{1}+\varphi_{1}^{(1)}$ of $ \varphi_{1} $, we apply the necessary condition $ \bar{f}\circ\bar{\varphi_{1}}\equiv 0 $ mod $ (b_0,\ldots,b_{29})^{2} $, and we solve for an unknown $ \varphi_{1}^{(1)} $ the equation:
$$ 0\equiv \bar{f}\circ\bar{\varphi_{1}}=(f+f^{(1)})(\varphi_{1}+\varphi_{1}^{(1)})=f\circ \varphi_{1}+(f\circ \varphi_{1}^{(1)}+f^{(1)}\circ \varphi_{1})\ \text{mod} \ (b_0,\ldots,b_{29})^{2}.$$
This leads to $ f\circ \varphi_{1}^{(1)}= -f^{(1)}\circ \varphi_{1} $, such that solving it for $ \varphi_{1}^{(1)} $ by matrix quotient gives the required perturbation of the first syzygy matrix $ \varphi_{1} $. Continuing through the remaining resolution maps, we can lift the entire resolution to first order in the same way.  In \cite{kensch}, an implementation of this algorithm is provided by the function \texttt{liftDeformationToFreeResolution}, which lifts a resolution to the first order deformed resolution.
\begin{theorem}
\label{tangentlocus}
Let $ 0\leq m\leq 4 $, and set $ k:=m+5 $. The stratum $ \mathcal{M}_{11,6}(k)\subset \mathcal{M}_{11} $ has an irreducible component $ H_k $ of expected dimension $ 30-k $. Moreover, at a general point $ P\in H_k $, $ \mathcal{M}_{11,6} $ is locally analytically a union of $ k $ smooth transversal branches. In other words, $ \mathcal{M}_{11,6} $ is a normal crossing divisor around the point $ P $.
\end{theorem} 
\begin{proof}
Consider the natural commutative diagram
\begin{displaymath}
\xymatrix{
\mathcal{V}_{9}^{4,5,m} \ar[r]^{\psi}\ar[d]_\phi &  \mathcal{H}_{C/\mathbb{P}^{10}}\ar[d]\\
\mathcal{M}_{11,6}(k) \ar@{^{(}->}[r] & \mathcal{M}_{11}}
\end{displaymath}
where $\phi $ takes the plane curve to its canonical model forgetting the embedding. 
Let $ H_k\subset\mathcal{M}_{11,6}(k)$ be the irreducible component containing the image points of curves lying in an irreducible component $ H\subset \mathcal{V}_{9}^{4,5,m} $ with expected dimension (see Theorem \ref{tangent}). We show that $ H_k$ is of expected dimension.\\
Let $C\subset \mathbb{P}^{10}$ be a canonical curve with $\ell$ extra syzygies, and let $\mathcal{C} \to (T,0)$ be its Kuranishi family. Then, the Koszul divisor $\mathcal{K}_{11}$ can be computed locally in this family as follows. One extends the minimal free resolution of the coordinate ring over $\mathbb{K}[x_0,\ldots,x_{10}]$ to a resolution over
$\mathcal{O}_{T,0}[x_0, \ldots, x_{10}]$.
The resulting complex will have a  $\ell \times \ell$ square submatrix with entries in $\mathcal{O}_{T,0}$. The determinant of this matrix defines the Koszul divisor restricted to the Kuranishi family. Due to Hirschowitz and Ramanan \cite{hirsch}, this divisor coincides with $5$ times the Brill-Noether divisor $\mathcal{M}_{11,6}$, that is $ \mathcal{K}_{11}=5\mathcal{M}_{11,6} $. Thus, the determinant of the matrix is a fifth power. Here, we compute the first order terms of this matrix for specific curves in various strata. \\
Let $ C\in H_k $ be the canoical model of a plane curve $ \Gamma\in H $ and for the general first order deformation family of $ C$, let $ M $ denote the $ 5k\times 5k $ submatrix of $ \overline{\varphi}_{4} $ in the deformed free resolution with linear entries in free deformation parameters $ b_0,\ldots,b_{29} $. In a minimal free resolution of $C$, the matrix defining the map $S(-6)^{5k}\longleftarrow S(-6)^{5k}$ is zero, hence the condition $ M=0 $ determines the space of the first order deformations of $C$ with extra syzygies of rank $ 5k $. \\
By means of the implemented function \texttt{verfiyAssertion(2)} in \cite{kenesh}, we can compute an explicit single example which shows that for exactly $ k $ linearly independent linear forms 
\[ l_{1},\ldots,l_{k}\in  \mathbb{K}[b_0,\ldots,b_{29}],\]
we have 
\[ \det \ M=l_{1}^{5}\cdot \ldots\cdot l_{k}^{5}. \] 
As the entries of the matrix $M$ are linear combinations of the $k$ independent forms $l_{1},\ldots, l_{k}$, one has $M=0$ if and only if $l_{1}= \cdots=l_{k}=0$. Moreover, identifying $T_C\mathcal{M}_{11,6}(k)$ with the space of first order deformations of $C$ with $k$ pencils, $T_C\mathcal{M}_{11,6}(k)$ is a subset of the space of the first order deformations of $C$ with extra syzygies of rank $5k$. Thus, since $\dim T_C\mathcal{M}_{11,6}(k)\geq 30-k$, the tangent space $T_C\mathcal{M}_{11,6}(k) $ is the zero locus of these linear forms, and is of codimension exactly $ k $ inside $ T_{C}\mathcal{M}_{11} $. Hence, $ H_k $ is an irreducible component of expected dimension $ 25-m $. Moreover, the equality $ \mathcal{K}_{11}=5\mathcal{M}_{11,6} $ implies that
 $$T_C\mathcal{M}_{11,6}=V(\ell_1)\cup \ldots \cup V(\ell_k),$$
which proves $ \mathcal{M}_{11,6}  $ at the point $ C $ is locally analytically union of $ k $ smooth branches.
\end{proof}
\begin{remark}
With the notation as above, under a change of basis, we can turn the matrix $ M $ to a block (or even a diagonal) matrix 
\[
M'=\begin{pmatrix} 
 B_1 & 0& \ldots & 0\\
 0& B_2 & \ldots &0\\
 \vdots  & \vdots & \ddots & \vdots\\
0  & 0 &  \ldots& B_k
\end{pmatrix}
\]   
such that for $ i=1,\ldots, k $ the non-zero block is $ B_i=A_iL_i $,
where $ A_i $ is an invertible $ 5\times 5 $ matrix with constant entries and $ L_i $ is the diagonal matrix with diagonal entries equal to $ l_i $. In fact, for $ i=1,\ldots,k $, let $ X_i $ be the scroll swept out by the pencil $ g_i $ on $ C $. Let $ M_i=(MV_i)^{t} $ be the $ 5\times 5k $ matrix, where $ V_i $ is the constant matrix defining the last map $\varphi_{i}=S(-6)^{5}\longrightarrow S(-6)^{5k} $ in the injective morphism of chain complexes from the resolution of $ X_i $ to the linear strand of a minimal resolution of $ C $. Set $ W_i:=\ker M_i $ and for $ j\in \lbrace 1,\ldots, k\rbrace $, let $ \overline{W}_{j}$ be the intersection of the modules $ W_i $'s by omitting $ W_j $. Since the pencils are mutually independent so that the corresponding scrolls contribute independently to the rank of the module $ \bar{S}(-6)^{5k}$, one can see $ \rank W_i=5(k-1) $, and a basis of $ W_i $ can be identified by columns of a constant matrix of size $ 5k\times 5(k-1) $. Moreover, we have $ \rank \overline{W}_{j}=5 $ such that a basis of the module $\overline{W}_{1}\oplus\ldots\oplus\overline{W}_{k} $ determines a $ 5k\times 5k $ invertible constant matrix. Using this invertible matrix for changing the basis of the space $ \bar{S}(-6)^{5k} $ turns the matrix $ M $ to a block matrix as above. To speed up our computations, we have used this presentation of $ M $ to compute its determinant. 
\end{remark}
\begin{theorem}
\label{g310}
The locus $  \mathcal{M}_{11,10}^{3} $ of genus $ 11 $ curves with a $ g^3_{10} $ is an irreducible component of $  \mathcal{M}_{11,6}(20)$ with expected dimension $ 25 $.
\end{theorem}
\begin{proof}
With the same arument as above, the theorem follows from computation of an explicit example (see \texttt{verfiyAssertion(6)} in \cite{kenesh}) which proves for five linearly independent linear forms $  l_{1},\ldots,l_{5} $ we have
 $$ T_C\mathcal{M}_{11,10}^{3}=T_C \mathcal{M}_{11,6}(20)=V(l_1,\ldots,l_5).$$
\end{proof}
\section{Unirational irreducible components}\label{sec:injectivity}
In this section, we prove that the so-constructed rational families of plane curves dominate an irreducible component of the strata $ \mathcal{M}_{11,6}(k) $ for $ k=5,\ldots,10 $. To this end, we count the number of moduli for these families, by computing the rank of the differential map between the tangent spaces.
\begin{theorem}
\label{degree9}
For $ 5\leq k\leq9 $, the stratum $ \mathcal{M}_{11,6}(k)$ has a  unirational irreducible component of expected dimension $ 30-k $. A general curve lying on this component arises from a degree $ 9 $ plane model with $ 4 $ ordinary triple and $ 5 $ ordinary double points which contains $ k-5 $ points among the ninth fixed point of the pencil of cubics passing through the $ 4 $ triple and $ 4 $ chosen double points.
\end{theorem}
\begin{proof}
With notations as in Theorem \ref{tangentlocus}, let $ \phi_{\vert H}:H\longrightarrow H_{k} $ be the natural map between the irreducible components of expected dimensions. To compute the dimension of $ \overline{\phi(H)} $, one has to compute the rank of the differential map 
$$ d\phi_{\Gamma}:T_{\Gamma}H\longrightarrow T_{C}H_{k},$$
at a smooth point $ C\in \phi(H)$. We recall that for $ m>0 $ the tangent space to $ \mathcal{V}_{9}^{4,5,m} $ at a point $ \Gamma $ is a subspace of $T_{\Gamma} \mathcal{V}_{9}^{4,5} $. Therefore, it suffices to show that $ \dim( \ker \ d\phi_{\Gamma})=8 $ for the case $ k=5 $. Considering the following commutative diagram of tangent maps
\begin{displaymath}
\xymatrix{
& &T_{\Gamma}H \ar[d]_{d\psi_{\Gamma}}\ar[r]^{d\phi_{\Gamma}}& T_{C}H_{k} \ar@{^{(}->}[d]\\
0 \ar[r]& \Hb^{0}(C,T_{\mathbb{P}^{10}}\vert_{C}) \ar[r] & \Hb^{0}(C,\mathcal{N}_{C/\mathbb{P}^{10}})\ar[r]&  \Hb^{1}(C,T_{C}) \ar[r]& 0
}
\end{displaymath}
our explicit computation of a single example (see \texttt{VerfiyAssertion(3)} in \cite{kenesh}) shows that the image of the map $ d\psi_{\Gamma} $ has exactly $ 8- $dimensional intersection with the image of $ \Hb^{0}(C,T_{\mathbb{P}^{10}}\vert_{C}) $ inside $ \Hb^{0}(C,\mathcal{N}_{C/\mathbb{P}^{10}}) $, which corresponds to the automorphisms of the projective plane. Therefore, the rational family of plane curves lying on the irreducible component $ H $ dominates an irreducible component of $ \mathcal{M}_{11,6}(k) $ with expected dimension.   
\end{proof}
\begin{theorem}
\label{degree10}
The stratum $ \mathcal{M}_{11,6}(10) $ has a unirational irreducible component of excess dimension $26$, where the curves arise from degree $ 8 $ plane models with $ 10 $ ordinary double points. More precisely, the locus $  \mathcal{M}_{11,8}^{2} $ of curves possessing a linear system $ g^{2}_{8} $ is a unirational irreducible component of $ \mathcal{M}_{11,6}(10) $ of expected dimension $ 26 $.
\end{theorem}
\begin{proof}
Let $ \mathcal{V}_{8}^{10} $ be the Severi variety of degree $ 8 $ plane curves with $ 10 $ ordinary double points. By classical results \cite{zariskiharris}, it is known that $ \mathcal{V}_{8}^{10} $ is smooth at each point and of pure dimension $ 34 $. Let $ \Gamma $ be a plane curve of degree $ 8 $ with $ 10 $ ordinary double points, and let $ C\in \mathcal{M}_{11,8}^{2}\subset \mathcal{M}_{11,6}(10) $ be its normalization. With the same argument as in the proof of \ref{tangentlocus} and \ref{degree9}, the theorem follows from the computation of an example which shows that for linear forms $ l_1,\ldots,l_{10} $ we have $ \dim T_C\mathcal{M}_{11}(10)=\dim V(l_1,\ldots,l_{10})=26 $ and furthermore the induced differential map is of full rank $ 26 $. The verification of this statement is implemented in the function \texttt{verifyAssertion(4)} in \cite{kenesh}.
\end{proof}
\begin{corollary}
Let $ \Gamma $ be a general plane curve of degree $ 8 $ with $ 10 $ ordinary double points, and let $ C\in \mathcal{M}_{11} $ be its normalization. Consider a deformation of $ C $ which preserves at least four pencils $ g^{1}_{6}$'s of the $ 10 $ existing pencils. Then, the deformation of $ C $ preserves the $ g^{2}_{8} $. In other words, a deformation of $ C $ which keeps at least four pencils $ g^{1}_{6}$'s lies still on the locus $ \mathcal{M}_{11,8}^{2} $.
\end{corollary}
\begin{proof}
By the above theorem, around a general point $ C\in \mathcal{M}_{11,8}^{2}$, the Brill--Noether divisor $ \mathcal{M}_{11,6} $ is locally a union of $ 10 $ branches defined by $ l_1\cdot \ldots \cdot l_{10}=0$. On the other hand, $ \codim T_C \mathcal{M}_{11,8}^{2}=\codim V (l_1,\ldots, l_{10})=4 $, such that any four of the linear forms are independent defining $ \mathcal{M}_{11,8}^{2} $ locally around $ C $. Therefore, a deformation of $ C $ which keeps at least four of $ g^{1}_{6}$'s lies still on the locus $ \mathcal{M}_{11,8}^{2} $. 
\end{proof}
\section{Further components}
\label{sec:Further components}
Having already described an irreducible unirational component of the strata $ \mathcal{M}_{11,6}(k) $ for $ k=5,\ldots,10 $, the first natural question is to ask about the irreducibility of these strata. If the answer is negative, then the question is how the other irreducible components arise. \\

Although one may mimic our pattern to find model of plane curves of higher degree with singular points of higher multiplicity, considering the degree $ 9 $ plane curves with $ 4 $ ordinary triple and $ 5 $ ordinary double points as our original model, our simple computations indicates that the models of higher degree are usually a Cremona transformation of this model with respect to three singular points. Therefore, considering models of different degrees and singularities, we have not found new elements in these strata. On the other hand, the study of syzygy schemes of curves lying on these strata leads to the following theorem which states the existence of further irreducible components.\\

\begin{theorem}
\label{othercomponent}
For $ 5\leq k\leq 8 $, the stratum $  \mathcal{M}_{11,6}(k) $ has at least two irreducible components both of expected dimension, along which $\mathcal{M}_{11,6} $ is generically a simple normal crossing divisor.
\end{theorem}
\begin{proof}
The proof relies on the syzygy schemes and our computation of tangent cone at a point C in $ H_k $.

 Consider $ \eta: \mathcal{W}_{11,6}^1 \longrightarrow\mathcal{M}_{11,6}\subset \mathcal M_{11}$
and let $C$ be a point in our unirational component $H_k \subset \mathcal{M}_{11,6}(k)$ for $ 6\leq k\leq 9 $. 
Then, by the Theorem \ref{tangentlocus}, the tangent cone of 
the Brill-Noether divisor $\mathcal{M}_{11,6} $ is defined by a product 
$l_1 \cdot\ldots\cdot l_k$ of $k$ linearly independent linear forms,
and $\mathcal{W}_{11,6}^1 \longrightarrow \mathcal{M}_{11,6} $ is locally around $C$ the normalization of 
$\mathcal{M}_{11,6} $. Let $f_1, \ldots, f_k$ be power series which define the $k$ branches of $\mathcal{M}_{11,6}$
in an analytic or \'{e}tale neighbourhood $U$ of $C \in \mathcal{M}_{11}$. Then 
$$f_i = l_i + \hbox{ higher order terms }$$
and the zero locus $V(f_i) \subset U$ has the following interpretation:
 $$V(f_i) \cong \{ (C',L'):  \ (C',L')\in\ U_i \},$$
where $ \eta^{-1}(U) = \bigcup_{i=1}^k U_i$ is the disjoint union of smooth $ 3g-4 $ dimensional manifolds with $(C,L_i) \in U_i$ such that $L_i$ denotes line bundle 
corresponding to the the $i-$th pencil $g^{1}_6$ on $C$ 
in some enumeration of the pencils $L_1, \ldots,L_k $  that we fix. \\
The submanifold $B_i=\{f_i=0\}$ then consists of deformations of $C$ induced by 
deformation of pair $(C,L_i)$, and for any family  
$\Delta \subset B_i$ the Kuranishi family restricted to $\Delta$ extends to a 
deformation of the pair $(C,L_i)$
$$
\begin{matrix}
 C &\subset &\mathcal C  & \quad   &   (C,L_i) &\subset &(\mathcal C,\mathcal L_i) \cr
  \downarrow && \downarrow && \downarrow && \downarrow \cr
 0 & \in & \Delta & \quad & 0 & \in& \Delta \cr
\end{matrix}
$$
Let $I \subset \{1,\ldots,k\}$ be any subset of cardinality $\ell\geq 5$ and $C' \in U$ be a point such that
$$C'\in \bigcap_{i\in I} V(f_i) \setminus \bigcup_{j\notin I} V(f_j).$$
Then, by Theorem \ref{tangentlocus}
$$C' \in \mathcal M_{11,6}(\ell) \setminus \mathcal M_{11,6}(\ell+1) $$
since the $l_i$ with $i \in I$
are linearly independent, $\mathcal M_{11,6}(\ell)$ is of codimension $\ell$ and $ \mathcal M_{11,6} $ is a normal crossing divisor around $C'$.

Now, we examine that whether or not $C'$ lies in our component $H_\ell$.
For this purpose, we deform the $L_i$ for $i \in I$ in a one-dimensional family of curves 
$$ \Delta=\lbrace C''\rbrace\subset \bigcap_{i\in I} V(f_i)$$
through $C$ and $C'$, which intersects  $  \bigcup_{j\notin I} V(f_j) $ only in the point $ C $. The syzygy schemes of the $C''\in \Delta$ forms an algebraic family
defined by the intersection of the deformed scrolls $X_i''$ swept out by the deformed line bundle $ L''_i $.
Thus by semicontinuity, the dimension of the syzygy scheme of $C''$ near $C \in \Delta$ is smaller or equal than the dimension of the syzygy scheme $ \bigcap X_i$, and in case of
 equality we should have $\deg(\bigcap X''_i)\leq \deg(\bigcap X_i)$. If we take special syzygy scheme of $ C'' $ corresponding to the syzygies of $ \bigcap_{j\in J}X_j'' $ then likewise we have the semicontinuity compare to $ \bigcap_{j\in J}X_j $. Therefore, for $ C'' $ to lie on $ H_l $ we need a subset $ J\subset I $ of cardinality $ 5 $ such that the syzygy scheme is a surface of degree $ 15 $ (see table \ref{table4}). By the Remark \ref{aandb}, this occurs only if we have $a=5$ and $b=0$. Thus, taking $I$ to be a subset of $ \{2,\ldots,5\} \cup \{6,\ldots,k\}$ we obtain a point $C'' \in \mathcal M_{11,6}(\ell) \setminus H_\ell$. This proves that for $5  \leq \ell \leq 8 $ the stratum $\mathcal{M}_{11,6}(\ell)$ has at least two components, one of which $H_\ell$ and the other a component containing $C''$.
\end{proof}
In paricular, considering the five smooth transversal branches of $ \mathcal{M}_{11,6} $ at a general point $ C $ of the irreducible conponent $ H_5\subset  \mathcal{M}_{11,6}(5)  $, we can deform $ C $ away from one of the branches, in a one-dimensional family of curves with $ 4 $ pencils, which proves the following.
\begin{theorem}
The stratum $ \mathcal{M}_{11,6}(4) $ has an irreducible component of expected dimension $ 26 $.
\end{theorem}
\begin{remark}
\label{aandb}
For the model of plane curve of degree $ 9 $ with nine pencils described in \ref{sec:Planar model description}, we have computed the dimension, degree and the Betti table of the syzyzgy schemes associated to different number $ 2\leq l\leq 9 $ of pencils $ g^{1}_{6}$'s. We recall that for a number of pencils indexed by a subset $ I\subset  \lbrace 1,\ldots,9 \rbrace  $, the associated syzygy scheme is the intersection $ \bigcap_{i\in I} X_i $ of the scrolls swept out by each of the pencils. Let $ 1\leq a\leq 5 $ be the number of chosen pencils which are induced by projection from the triple points or the pencil of conics. Likewise, let $ 1\leq b\leq 4 $ be the number of chosen pencils arised from the pencil of cubics through the certain number of points. In the following tables, and for a specific genus $ 11 $ curve possessing nine pencils, we have listed the numerical data of the plausible syzygy schemes arised form different number $ l=a+b\geq 2 $ of the existing pencils $ g^{1}_{6}$'s. In \cite{kenesh}, one can compute an example of such a curve over a finite field of characteristic $ p $, by running the function \texttt{random6gonalGenus11Curvekpencil(p,9)}. In particular, the function \texttt{verifyAsserion(5)} provides the explicit equation of our specific curve and the collection of the nine scrolls. In the columns "dim" and "deg" we have marked the possible dimension and the degree of the corresponding syzygy schemes for this specific curve. Based on our experiments, it turns out that the values only depend on the numbers $ a $ and $ b $ of the chosen pencils.\\
\begin{longtable}{ |l|l|l|l|l|l| }
\multicolumn{6}{ c } {}\\
\hline
 & $  $ & $ \dim $ &  $ \deg $ & \text{genus} & \hspace*{2.5cm}$ \text{ Betti table} $\\ 
\hline
$\ \  a=0 $ & $\ \ b=2 $ & $\  2 $ &  $ 18$ &  & \ \small{\begin{tabular}{|c c c c c c c c c }
\noalign{\vskip 2mm}
  \hline   
  $1$ & $.$ & $ . $& $.$ & $.$ & $.$& $ . $ &$ . $&$ . $\\ 
  $.$ & $27$ & $ 96 $&$127$ & $48$ & $10$& $ . $& $ . $& $ . $\\ 
  $.$ &$.$ & $1$ & $ 48 $& $220$ & $288$ &$189$ &$ 64 $&$ 9 $ \\
\noalign{\vskip 2mm}           
\end{tabular}}\\
\hline
$\ \  a=1 $ & $ \ \ b=1 $ & $\  2 $ &  $ 18$ &  &\ \small{\begin{tabular}{|c c c c c c c c c }
\noalign{\vskip 2mm} 
  \hline
  $1$ & $.$ & $ . $& $.$ & $.$ & $.$& $ . $ &$ . $&$ . $\\ 
  $.$ & $27$ & $ 96 $&$127$ & $48$ & $10$& $ . $& $ . $& $ . $\\ 
  $.$ &$.$ & $1$ & $ 48 $& $220$ & $288$ &$189$ &$ 64 $&$ 9 $ \\ 
  \noalign{\vskip 2mm}            
\end{tabular} }\\
\hline
$\ \  a=2 $ & $\ \ b=0 $ & $\  2 $ &  $ 18$ &  &\  \small{\begin{tabular}{|c c c c c c c c c }
\noalign{\vskip 2mm} 
  \hline
  $1$ & $.$ & $ . $& $.$ & $.$ & $.$& $ . $ &$ . $&$ . $\\ 
  $.$ & $27$ & $ 96 $&$127$ & $48$ & $10$& $ . $& $ . $& $ . $\\ 
  $.$ &$.$ & $ 1$& $48$ & $220$ &$288$ &$ 189 $&$ 64 $& $9 $ \\   
  \noalign{\vskip 2mm}        
\end{tabular}}\\
\hline 
\caption{\label{table1} Numerical data of possible syzygy schemes with $ a+b=2 $.}
\end{longtable}
\newpage
\begin{longtable}{ |l|l|l|l|l|l| }
\multicolumn{6}{ c } {}\\
\hline
 & $  $ & $ \dim $ &  $ \deg $ & \text{genus} & \hspace*{2.5cm}$ \text{ Betti table} $\\ 
\hline
$ \ \ a=0 $ & $ \ \ b=3 $ & $\  1 $ & 
$ 21 $
&   $ 12 $ & \
 \small{ \begin{tabular}{|c c c c c c c c c c}
 \noalign{\vskip 2mm} 
 \hline
  $1$ & $.$ & $ . $& $.$ & $.$ & $.$& $ . $ &$ . $&$ . $&$ . $\\ 
  $.$ & $35$ & $ 151 $&$279$ & $207$ & $15$& $ . $& $ . $& $ . $& $ . $\\ 
  $.$ &$.$ & $.$ & $ 3 $& $141$ & $414$ &$399$ &$ 196 $&$ 45 $&$ 1 $\\  
  $.$ &$.$ & $.$ & $ . $& $.$ & $.$ &$.$ &$ . $&$ . $&$ 1 $ \\
  \noalign{\vskip 2mm}             
\end{tabular}} \\ 
\hline
$ \ a=1 $ & $ \ b=2 $ & $\  1 $ &  $ 20$ & $ \ 11 $ & \small{ 
 \begin{tabular}{|c c c c c c c c c c}
 \noalign{\vskip 2mm} 
 \hline
  $1$ & $.$ & $ . $& $.$ & $.$ & $.$& $ . $ &$ . $&$ . $&$ . $\\ 
  $.$ & $36$ & $ 160 $&$315$ & $288$ & $45$& $ . $& $ . $& $ . $& $ . $\\ 
  $.$ &$.$ & $.$ & $ . $& $45$ & $288$ &$315$ &$ 160 $&$ 36 $&$ . $\\  
  $.$ &$.$ & $.$ & $ . $& $.$ & $.$ &$.$ &$ . $&$ . $&$ 1 $  \\
  \noalign{\vskip 2mm}            
\end{tabular}}\\   
\hline
$ \ a=2$ & $ \ \ b=1 $ & $\  1 $ &  $ 21$ & $ \ 12 $ &  \small{ \begin{tabular}{|c c c c c c c c c c}
\noalign{\vskip 2mm} 
 \hline
  $1$ & $.$ & $ . $& $.$ & $.$ & $.$& $ . $ &$ . $&$ . $&$ . $\\ 
  $.$ & $35$ & $ 151 $&$279$ & $210$ & $30$& $ . $& $ . $& $ . $& $ . $\\ 
  $.$ &$.$ & $.$ & $ . $& $6$ & $156$ &$414$ &$ 399 $&$ 45 $&$ 1 $\\  
  $.$ &$.$ & $.$ & $ . $& $.$ & $.$ &$.$ &$ . $&$ . $&$ 1 $ \\
  \noalign{\vskip 2mm}             
\end{tabular}}\\
\hline 
$\ \  a=3 $ & $\ \ b=0 $ & $\  2 $ &  $ 16$ &  &\  \small{\begin{tabular}{|c c c c c c c c c }
\noalign{\vskip 2mm} 
  \hline
  $1$ & $.$ & $ . $& $.$ & $.$ & $.$& $ . $ &$ . $&$ . $\\ 
  $.$ & $29$ & $ 112 $&$182$ & $113$ & $15$& $ . $& $ . $& $ . $\\ 
  $.$ &$.$ & $.$& $1$ & $ 85 $& $176$ & $133$ &$48$ &$ 7 $ \\ 
  \noalign{\vskip 2mm}          
\end{tabular}}\\
\hline 
\caption{\label{table2} Numerical data of possible syzygy schemes with $ a+b=3 $.}
\end{longtable}
\begin{longtable}{ |l|l|l|l|l|l| }
\multicolumn{6}{ c } {}\\
\hline
 & $  $ & $ \dim $ &  $ \deg $ & \text{genus} & \hspace*{2.5cm}$ \text{ Betti table} $\\ 
\hline
$ \ a=0 $ & $ \ b=4 $ & $\  1 $ &  $ 20$ & $ \ 11 $ & \small{ 
 \begin{tabular}{|c c c c c c c c c c}
 \noalign{\vskip 2mm} 
 \hline
  $1$ & $.$ & $ . $& $.$ & $.$ & $.$& $ . $ &$ . $&$ . $&$ . $\\ 
  $.$ & $36$ & $ 160 $&$315$ & $288$ & $45$& $ . $& $ . $& $ . $& $ . $\\ 
  $.$ &$.$ & $.$ & $ . $& $45$ & $288$ &$315$ &$ 160 $&$ 36 $&$ . $\\  
  $.$ &$.$ & $.$ & $ . $& $.$ & $.$ &$.$ &$ . $&$ . $&$ 1 $  \\
  \noalign{\vskip 2mm}            
\end{tabular}}\\   
\hline
$ \ a=1 $ & $ \ b=3 $ & $\  1 $ &  $ 20$ & $ \ 11 $ & \small{ 
 \begin{tabular}{|c c c c c c c c c c}
 \noalign{\vskip 2mm} 
 \hline
  $1$ & $.$ & $ . $& $.$ & $.$ & $.$& $ . $ &$ . $&$ . $&$ . $\\ 
  $.$ & $36$ & $ 160 $&$315$ & $288$ & $45$& $ . $& $ . $& $ . $& $ . $\\ 
  $.$ &$.$ & $.$ & $ . $& $45$ & $288$ &$315$ &$ 160 $&$ 36 $&$ . $\\  
  $.$ &$.$ & $.$ & $ . $& $.$ & $.$ &$.$ &$ . $&$ . $&$ 1 $  \\
  \noalign{\vskip 2mm}            
\end{tabular}}\\   
\hline
$ \ a=2 $ & $ \ b=2 $ & $\  1 $ &  $ 20$ & $ \ 11 $ & \small{ 
 \begin{tabular}{|c c c c c c c c c c}
 \noalign{\vskip 2mm} 
 \hline
  $1$ & $.$ & $ . $& $.$ & $.$ & $.$& $ . $ &$ . $&$ . $&$ . $\\ 
  $.$ & $36$ & $ 160 $&$315$ & $288$ & $45$& $ . $& $ . $& $ . $& $ . $\\ 
  $.$ &$.$ & $.$ & $ . $& $45$ & $288$ &$315$ &$ 160 $&$ 36 $&$ . $\\  
  $.$ &$.$ & $.$ & $ . $& $.$ & $.$ &$.$ &$ . $&$ . $&$ 1 $  \\
  \noalign{\vskip 2mm}            
\end{tabular}}\\   
\hline 
$ \ \ a=3 $ & $ \ \ b=1 $ & $\  1 $ & 
$ 21 $
&   $ \ 12 $ & 
 \small{ \begin{tabular}{|c c c c c c c c c c}
 \noalign{\vskip 2mm} 
 \hline
  $1$ & $.$ & $ . $& $.$ & $.$ & $.$& $ . $ &$ . $&$ . $&$ . $\\ 
  $.$ & $35$ & $ 151 $&$279$ & $210$ & $30$& $ . $& $ . $& $ . $& $ . $\\ 
  $.$ &$.$ & $.$ & $ . $& $6$ & $156$ &$414$ &$ 399 $&$ 45 $&$ 1 $\\  
  $.$ &$.$ & $.$ & $ . $& $.$ & $.$ &$.$ &$ . $&$ . $&$ 1 $ \\
  \noalign{\vskip 2mm}             
\end{tabular}} \\ 
\hline 
$\ \  a=4 $ & $\ \ b=0 $ & $\  2 $ &  $ 15$ &  &\  \small{\begin{tabular}{|c c c c c c c c c }
\noalign{\vskip 2mm} 
  \hline
  $1$ & $.$ & $ . $& $.$ & $.$ & $.$& $ . $ &$ . $&$ . $\\ 
  $.$ & $30$ & $ 120 $&$210$ & $169$ & $25$& $ . $& $ . $& $ . $\\ 
  $.$ &$.$ & $ . $&$1$ & $ 25 $& $120$ & $105$ &$40$ &$ 6 $ \\ 
  \noalign{\vskip 2mm}          
\end{tabular}}\\
\hline 
\caption{\label{table3} Numerical data of possible syzygy schemes with $ a+b=4 $.}
\end{longtable}
\newpage
\begin{longtable}{ |l|l|l|l|l|l| }
\multicolumn{6}{ c } {}\\
\hline
 & $  $ & $ \dim $ &  $ \deg $ & \text{genus} & \hspace*{2.5cm}$ \text{ Betti table} $\\ 
\hline
$ \ a=1 $ & $ \ b=4 $ & $\  1 $ &  $ 20$ & $ \ 11 $ & \small{ 
 \begin{tabular}{|c c c c c c c c c c}
 \noalign{\vskip 2mm} 
 \hline
  $1$ & $.$ & $ . $& $.$ & $.$ & $.$& $ . $ &$ . $&$ . $&$ . $\\ 
  $.$ & $36$ & $ 160 $&$315$ & $288$ & $45$& $ . $& $ . $& $ . $& $ . $\\ 
  $.$ &$.$ & $.$ & $ . $& $45$ & $288$ &$315$ &$ 160 $&$ 36 $&$ . $\\  
  $.$ &$.$ & $.$ & $ . $& $.$ & $.$ &$.$ &$ . $&$ . $&$ 1 $  \\
  \noalign{\vskip 2mm}            
\end{tabular}}\\   
\hline
$ \ a=2 $ & $ \ b=3 $ & $\  1 $ &  $ 20$ & $ \ 11 $ & \small{ 
 \begin{tabular}{|c c c c c c c c c c}
 \noalign{\vskip 2mm} 
 \hline
  $1$ & $.$ & $ . $& $.$ & $.$ & $.$& $ . $ &$ . $&$ . $&$ . $\\ 
  $.$ & $36$ & $ 160 $&$315$ & $288$ & $45$& $ . $& $ . $& $ . $& $ . $\\ 
  $.$ &$.$ & $.$ & $ . $& $45$ & $288$ &$315$ &$ 160 $&$ 36 $&$ . $\\  
  $.$ &$.$ & $.$ & $ . $& $.$ & $.$ &$.$ &$ . $&$ . $&$ 1 $  \\
  \noalign{\vskip 2mm}            
\end{tabular}}\\   
\hline
$ \ a=3 $ & $ \ b=2 $ & $\  1 $ &  $ 20$ & $ \ 11 $ & \small{ 
 \begin{tabular}{|c c c c c c c c c c}
 \noalign{\vskip 2mm} 
 \hline
  $1$ & $.$ & $ . $& $.$ & $.$ & $.$& $ . $ &$ . $&$ . $&$ . $\\ 
  $.$ & $36$ & $ 160 $&$315$ & $288$ & $45$& $ . $& $ . $& $ . $& $ . $\\ 
  $.$ &$.$ & $.$ & $ . $& $45$ & $288$ &$315$ &$ 160 $&$ 36 $&$ . $\\  
  $.$ &$.$ & $.$ & $ . $& $.$ & $.$ &$.$ &$ . $&$ . $&$ 1 $  \\
  \noalign{\vskip 2mm}            
\end{tabular}}\\   
\hline 
$ \ a=4 $ & $ \ b=1 $ & $\  1 $ &  $ 21$ & $ \ 12 $ &
 \small{ \begin{tabular}{|c c c c c c c c c c}
 \noalign{\vskip 2mm} 
 \hline
  $1$ & $.$ & $ . $& $.$ & $.$ & $.$& $ . $ &$ . $&$ . $&$ . $\\ 
  $.$ & $35$ & $ 151 $&$279$ & $210$ & $30$& $ . $& $ . $& $ . $& $ . $\\ 
  $.$ &$.$ & $.$ & $ . $& $6$ & $156$ &$414$ &$ 399 $&$ 45 $&$ 1 $\\  
  $.$ &$.$ & $.$ & $ . $& $.$ & $.$ &$.$ &$ . $&$ . $&$ 1 $ \\
  \noalign{\vskip 2mm}             
\end{tabular}} \\ 
\hline 
$\ \  a=5 $ & $\ \ b=0 $ & $\  2 $ &  $ 15$ &  &\  \small{\begin{tabular}{|c c c c c c c c c }
\noalign{\vskip 2mm} 
  \hline
  $1$ & $.$ & $ . $& $.$ & $.$ & $.$& $ . $ &$ . $&$ . $\\ 
  $.$ & $30$ & $ 120 $&$210$ & $169$ & $25$& $ . $& $ . $& $ . $\\ 
  $.$ &$.$ & $ . $&$1$ & $ 25 $& $120$ & $105$ &$40$ &$ 6 $ \\ 
  \noalign{\vskip 2mm}        
\end{tabular}}\\
\hline 
\caption{\label{table4} Numerical data of possible syzygy schemes with $ a+b=5 $.}
\end{longtable}
\begin{longtable}{ |l|l|l|l|l|l| }
\multicolumn{6}{ c } {}\\
\hline
 & $  $ & $ \dim $ &  $ \deg $ & \text{genus} & \hspace*{2.5cm}$ \text{ Betti table} $\\ 
\hline
$ \ a=2 $ & $ \ b=4 $ & $\  1 $ &  $ 20$ & $ \ 11 $ & \small{ 
 \begin{tabular}{|c c c c c c c c c c}
 \noalign{\vskip 2mm} 
 \hline
  $1$ & $.$ & $ . $& $.$ & $.$ & $.$& $ . $ &$ . $&$ . $&$ . $\\ 
  $.$ & $36$ & $ 160 $&$315$ & $288$ & $45$& $ . $& $ . $& $ . $& $ . $\\ 
  $.$ &$.$ & $.$ & $ . $& $45$ & $288$ &$315$ &$ 160 $&$ 36 $&$ . $\\  
  $.$ &$.$ & $.$ & $ . $& $.$ & $.$ &$.$ &$ . $&$ . $&$ 1 $  \\
  \noalign{\vskip 2mm}            
\end{tabular}}\\   
\hline
$ \ a=3 $ & $ \ b=3 $ & $\  1 $ &  $ 20$ & $ \ 11 $ & \small{ 
 \begin{tabular}{|c c c c c c c c c c}
 \noalign{\vskip 2mm} 
 \hline
  $1$ & $.$ & $ . $& $.$ & $.$ & $.$& $ . $ &$ . $&$ . $&$ . $\\ 
  $.$ & $36$ & $ 160 $&$315$ & $288$ & $45$& $ . $& $ . $& $ . $& $ . $\\ 
  $.$ &$.$ & $.$ & $ . $& $45$ & $288$ &$315$ &$ 160 $&$ 36 $&$ . $\\  
  $.$ &$.$ & $.$ & $ . $& $.$ & $.$ &$.$ &$ . $&$ . $&$ 1 $  \\
  \noalign{\vskip 2mm}            
\end{tabular}}\\   
\hline
$ \ a=4 $ & $ \ b=2 $ & $\  1 $ &  $ 20$ & $ \ 11 $ & \small{ 
 \begin{tabular}{|c c c c c c c c c c}
 \noalign{\vskip 2mm} 
 \hline
  $1$ & $.$ & $ . $& $.$ & $.$ & $.$& $ . $ &$ . $&$ . $&$ . $\\ 
  $.$ & $36$ & $ 160 $&$315$ & $288$ & $45$& $ . $& $ . $& $ . $& $ . $\\ 
  $.$ &$.$ & $.$ & $ . $& $45$ & $288$ &$315$ &$ 160 $&$ 36 $&$ . $\\  
  $.$ &$.$ & $.$ & $ . $& $.$ & $.$ &$.$ &$ . $&$ . $&$ 1 $  \\
  \noalign{\vskip 2mm}            
\end{tabular}}\\   
\hline 
$ \ a=5 $ & $ \ b=1 $ & $\  1 $ &  $ 21$ & $ \ 12 $ &
 \small{ \begin{tabular}{|c c c c c c c c c c}
 \noalign{\vskip 2mm} 
 \hline
   $1$ & $.$ & $ . $& $.$ & $.$ & $.$& $ . $ &$ . $&$ . $&$ . $\\ 
  $.$ & $35$ & $ 151 $&$279$ & $210$ & $30$& $ . $& $ . $& $ . $& $ . $\\ 
  $.$ &$.$ & $.$ & $ . $& $6$ & $156$ &$414$ &$ 399 $&$ 45 $&$ 1 $\\  
  $.$ &$.$ & $.$ & $ . $& $.$ & $.$ &$.$ &$ . $&$ . $&$ 1 $ \\
  \noalign{\vskip 2mm}             
\end{tabular}} \\ 
\hline 
\caption{\label{table5} Numerical data of possible syzygy schemes with $ a+b=6 $.}
\end{longtable}
\newpage

\begin{longtable}{ |l|l|l|l|l|l| }
\multicolumn{6}{ c } {}\\
\hline
 & $  $ & $ \dim $ &  $ \deg $ & \text{genus} & \hspace*{2.5cm}$ \text{ Betti table} $\\ 
\hline
$ \ a$ & $ \ b $ & $\  1 $ &  $ 20$ & $ \ 11 $ & \small{ 
 \begin{tabular}{|c c c c c c c c c c}
 \noalign{\vskip 2mm} 
 \hline
  $1$ & $.$ & $ . $& $.$ & $.$ & $.$& $ . $ &$ . $&$ . $&$ . $\\ 
  $.$ & $36$ & $ 160 $&$315$ & $288$ & $45$& $ . $& $ . $& $ . $& $ . $\\ 
  $.$ &$.$ & $.$ & $ . $& $45$ & $288$ &$315$ &$ 160 $&$ 36 $&$ . $\\  
  $.$ &$.$ & $.$ & $ . $& $.$ & $.$ &$.$ &$ . $&$ . $&$ 1 $  \\
  \noalign{\vskip 2mm}            
\end{tabular}}\\   
\hline
\caption{\label{table6} Numerical data of possible syzygy schemes with $ a+b\geq 7$.}
\end{longtable}
\end{remark}

\thanks{Hanieh Keneshlou: \texttt{keneshlo@mis.mpg.de}\\
\thanks{Frank-Olaf Schreyer: \texttt{schreyer@math.uni-sb.de}
\end{document}